\documentclass[15 pt]{amsart}
\usepackage{amsfonts, amsmath, amssymb, amscd, amsbsy, amscd, amsgen, amsopn, amstext, amsxtra}
\usepackage[latin5]{inputenc}
\usepackage{xcolor}

\newtheorem{thm}{Theorem}[section]
\newtheorem{cor}[thm]{Corollary}

\newtheorem{lem}[thm]{Lemma}

\newtheorem{pro}[thm]{Proposition}
\theoremstyle{definition}

\theoremstyle{question}

\theoremstyle{remark}

\theoremstyle{example}
\newtheorem{example}[thm]{Example}

\newcommand{\FF}{\mathbb{F}}

\def\ex{{\rm exp}}

\begin{document}

\title{On the number of non-$G$-equivalent minimal abelian codes }

\author{Fatma Altunbulak  Aksu and \.Ipek Tuvay}

\address{Department of Mathematics, Mimar Sinan Fine Arts University, \c{S}i\c{s}li, Istanbul, Turkey}

\email{fatma.altunbulak@msgsu.edu.tr, ipek.tuvay@msgsu.edu.tr}


\thanks{}


\date{\today}

\dedicatory{}

\commby{}

\begin{abstract} Let $G$ be a finite abelian group. Ferraz, Guerreiro and Polcino Milies prove that 
the number of  $G$-equivalence classes of minimal abelian codes is equal to the number 
of $G$-isomorphism classes of subgroups for which corresponding quotients are cyclic. 
In this article, we prove that the notion of $G$-isomorphism  is equivalent to the notion of 
isomorphism on the set of all subgroups $H$ of $G$  with the property that $G/H$ is cyclic. As an 
application, we calculate the number of non-$G$-equivalent minimal abelian codes for some 
specific family of abelian groups. We also prove that the number of non-$G$-equivalent minimal 
abelian codes is equal to the number of divisors of the exponent of $G$ if and only if for each prime 
$p$ dividing the order of $G$, the Sylow $p$-subgroups of $G$ are homocyclic.
\end{abstract}

\maketitle

\section{Introduction}

Due to Berman \cite{Berman} and MacWilliams \cite{MacWilliams1},
an {\bf abelian code over a field} is defined to be an ideal in a finite group algebra of an abelian group. 
An abelian code is said to be {\bf minimal} if the
corresponding ideal is minimal in the set of all ideals of the group algebra. 
Let $G$ be a finite abelian group and $\mathbb{F}$ a finite field of characteristic coprime to the order of $G$. 
Under these conditions, Maschke's Theorem says that every abelian code is a direct sum of 
minimal abelian codes. 
Moreoever as defined in \cite{Miller}, two abelian codes 
$I$ and $J$ are called {\bf $G$-equivalent} if there is a group automorphism 
$\varphi:G\rightarrow G$ whose linear extension to the group algebra maps $I$ onto $J$. 
It is easy to see that $G$-equivalent codes have the same weight distribution. However, the converse 
is not true (see Proposition IV.2 in \cite{FGM} ). Therefore, knowing the number of $G$-equivalence classes 
of minimal abelian codes tells us a lot about the nature of codes that can be defined using the group algebra $\mathbb{F}G$.

A one-to-one correspondence between $G$-equivalence classes of minimal abelian codes and $G$-isomorphism classes of
cocyclic subgroups of $G$ is established by Ferraz, Guerreiro and Polcino Milies. (For the details see Proposition III.2, Proposition III.7 
and Proposition III.8 in \cite{FGM} ). According to \cite{FGM}, two subgroups $H$ and 
$K$ of $G$ are called {\bf $G$-isomorphic} if there is an automorphism $\varphi$
of $G$ which maps $H$ onto $K$. A subgroup $L\leq G$  is called a {\bf cocyclic 
subgroup of $G$} if $G/L$ is cyclic. Note that this definition is not the same definition 
as in \cite{FGM}. We take into account $G$ itself also as a cocyclic subgroup to count the minimal 
abelian code which corresponds to the subgroup $G$ of $G$.
From the definition, it is clear that if two subgroups of $G$ are $G$-isomorphic, then they are isomorphic. 
However, the converse of this statement is not true for arbitrary subgroups of $G$. We observe that the notion of $G$-isomorphism 
is equivalent to the notion of isomorphism on the set of cocyclic subgroups of $G$ as follows.

\begin{pro}\label{G-iso} Let $G$ be a finite abelian group and let $H$, $K$ be cocyclic subgroups of $G$. Then $H$
and $K$ are $G$-isomorphic if and only if they are isomorphic.
\end{pro}
This proposition, together with Proposition III.2, Proposition
III.7 and Proposition III.8 in \cite{FGM} leads us to write the following theorem.

\begin{thm}\label{rewrite}  Let $G$ be a finite abelian group. The number of non-$G$-equivalent 
minimal abelian codes over $\FF$ is equal to the number of isomorphism classes of  cocyclic subgroups of $G$.
\end{thm}

Let $\eta({\FF G})$ denote the number of non-$G$-equivalent minimal abelian codes over $\FF $. 
As an application of Theorem \ref{rewrite}, we prove the following results. Among these, the first result is 
the following.

\begin{thm}\label{HxH}
Let $H$ be a finite abelian group and let $G$ be a direct product of finite number of copies of $H$. Then 
we have that $\eta({\FF G})$=$\eta({\FF H})$.
\end{thm}

We observe that under an assumption on the exponent of the direct factors, multiplying a finite abelian group
by a homocyclic group does not change the number $\eta({\FF G}).$ Here a {\bf homocyclic group}
is a direct product of pairwise isomorphic cyclic groups.

\begin{thm}\label{HxK}
Let $K$ be a finite homocyclic group and $H$ a finite abelian group  such that $\exp(K)=\exp(H)$. If $G=K \times H$,
then we have that $\eta({\FF G})$=$\eta({\FF H})$.
\end{thm}

As emphasized in \cite{Miller}, the codes arising from the group algebra $\FF_2 (C_m \times C_n)$, 
where $m$ and $n$ are positive odd integers, 
are referred as two-dimensional linear recurring arrays, linear recurring planes or two-dimensional 
cyclic codes in the works \cite{NF} and \cite{NMIF}. These codes are related 
to the problem of constructing perfect maps and have applications to $x$-ray photography. In \cite[Theorem 3.6]{Miller}, 
it is stated that the number of non-equivalent minimal codes of $\FF_2 (C_m \times C_n)$ is equal to
the number of divisors of the exponent of the corresponding group. Ferraz, Guerreiro and Polcino Milies point out that this result is not true by
calculating the number of non-equivalent minimal codes of $\FF_2 (C_{p^n} \times C_{p})$ as $2n$ where $p$ is an odd prime.
(see \cite[Propostion IV.3]{FGM}). The following theorem generalizes this result.

\begin{thm}\label{exactnumber}
If $G=C_{p^n}\times C_{p^m}$ and $n>m$, then $\eta({\FF G})=(n-m+1)(m+1)$.
\end{thm}

As a corollary we obtain the following result.

\begin{cor} Let $n$ be a positive integer such that $n={p_{1}}^{k_{1}}{p_2}^{k_2}\dots{p_t}^{k_t} $ 
where $p_i$'s are distinct prime numbers and $k_i$'s are positive integers. Then for 
$G=C_{n^l}\times C_{n^s}$ where $n,l,s$ are positive integers and $l>s$ we have that
$\eta(\FF G)=\prod_{i=1}^t(k_i l-k_i s+1)(k_i s+1)$.
\end{cor}

In \cite{Miller}, for an abelian group $G$ of odd order, it is proved that the
number of non-$G$-equivalent minimal abelian codes over $\mathbb{F}_2$
is equal to the number of divisors of the exponent of $G$. In \cite{FGM}, it is
shown that this statement is not true and moreover it is
shown that if $G$ is homocyclic, the number of non-$G$-equivalent minimal abelian
codes over $\mathbb{F}$ is equal to the number of divisors of exponent of
$G$ (see Theorem V.6 in \cite{FGM}).
In the following theorem, we extend this result and give a characterization of an abelian group whose number of
non-equivalent minimal codes is equal to the number of divisors of its exponent.

\begin{thm}\label{main} Let $G$ be a finite abelian group and $\mathbb{F}$ a finite field of characteristic coprime to order of $G$. The number of non-$G$-equivalent minimal abelian codes over $\mathbb{F}$ is equal to the number of divisors of exponent of $G$ if and only if for each prime $p$ dividing the order of $G$, the Sylow $p$-subgroups of $G$ are homocyclic.
\end{thm}

Note that Theorem V.6 in \cite{FGM} now follows from the Theorem \ref{main} as a corollary.

The structure of the paper is as follows. In section 2, we 
establish some resullts on isomorphic cocyclic subgroups of a finite abelian group and 
give the proofs of Proposition \ref{G-iso} and 
Theorem \ref{rewrite}, Theorem \ref{HxH} and 
Theorem \ref{HxK} . We also present some important examples related to Theorem \ref{HxK}. In section 3, we 
prove Theorem \ref{exactnumber}. In section 4, we present the proof of Theorem \ref{main}.

\section{Proof of Proposition \ref{G-iso} and its consequences}

It is not very easy to determine whether two subgroups of a given group $G$
are $G$-isomorphic or not. We begin by showing  
that isomorphisms between cocyclic subgroups of $G$ can be extended to an  automorphism of $G$. 
Then, we continue to prove some other facts to give a proof for Proposition \ref{G-iso}.

\begin{lem}\label{any-z} Let $G$ be an abelian $p$-group of exponent $p^n$ and $H$ a cocyclic subgroup of $G$. 
If the exponent of $H$ is strictly less than $p^n$, then for any $z\in G$ of order $p^n$, $zH$ generates the cyclic group $G/H$.
\end{lem}

\begin{proof} Let $z$ be any element in  $G$ of order $p^n$. As $\ex(H)<p^n$, 
we have that $z\notin H$, so $zH$ is a non-trivial element in $G/H$. Since $H$ is a cocyclic subgroup of $G$, 
$G/H \cong C_{p^m}$ for some $1\leq m\leq n$. Then 
there exists some $b\in G$ such that $G/H=\langle bH \rangle$. 
Since $zH$ is a non-trivial element in $G/H$, we have that 
$zH=(bH)^{p^i}$ for some $i$ where $0\leq i<m$. 
This implies that $(b^{p^i})^{-1}z\in H$. As $\ex(H)<p^n$, 
this is possible only if $z=b^{p^i}$. Since $z$ has order $p^n$, $i$ should be equal to zero 
so that $z=b$ 
and hence $G/H=\langle zH \rangle$.
\end{proof}

\begin{lem} \label{forsome-x}Let $G$ be an abelian $p$-group of exponent $p^n$ and $H<G$ 
a cocyclic subgroup of $G$. Then there exists an element $x\in G$ of order $p^n$ so that $G/H=\langle xH \rangle$.
\end{lem}

\begin{proof} There are two cases to consider.\\
\textbf{Case 1:} $\ex(H)<p^n$: Follows from Lemma \ref{any-z}.\\
\textbf{Case 2:} $\ex(H)=p^n$: As $H$ is cocyclic we have that $G/H\cong C_{p^m}$ for some $1\leq m\leq n.$ Then 
there exists $a\in G-H$ so that $G/H=\langle aH\rangle$ and then $a^{p^m}\in H$ 
and $m$ is the smallest integer satisfying this property. 
If the order of $a^{p^{m}}$ is equal to $p^{n-m}$, then one can take $x=a$.  
If the order of $a^{p^{m}}$ is less than $p^{n-m}$, as $\ex(H)=p^n,$ there exists some 
$y\in H$ of order $p^n$ so one can take $x=ya$. It is clear that $xH$ generates $G/H$. 

\end{proof}

\begin{lem} \label{isoquotient} Let $G$ be an abelian $p$-group of exponent $p^n$. 
Assume that $H$ and $K$  are subgroups of $G$ such that for some $x,y\in G$ of order $p^n$, 
we have that $G/H= \langle xH \rangle\cong C_{p^m}\cong G/K=\langle yK \rangle$ where $0\leq m\leq n$.
Then, we have that $H/ \langle x^{p^m} \rangle\cong K/ \langle y^{p^m} \rangle$. 
\end{lem}

\begin{proof} As $\ex(G)=p^n$ and $|x|=|y|=p^n$, 
we can write $$G= \langle x\rangle \times A = \langle y \rangle \times B$$ where $A \cong B$, so that  
$G/\langle x \rangle\cong G/\langle y \rangle$.
On the other hand, Second Isomorphism Theorem
gives us that
$$G/\langle x \rangle=\langle x\rangle H/\langle x\rangle\cong H/H\cap \langle x\rangle =H/\langle x^{p^m}\rangle$$ and similarly
$G/\langle y\rangle=\langle y\rangle K/\langle y \rangle\cong K/K\cap \langle y\rangle=K/\langle y^{p^m}\rangle$. 
Therefore, we deduce that 
$$H/\langle x^{p^m}\rangle\cong K/\langle y^{p^m}\rangle.$$

\end{proof} 

\begin{lem}\label{nongenerator} Let $H$ be a finite abelian $p$-group and let $x$ be a non-generator of $H$. 
Then there exists a generator $a$ of $H$ such that $\langle x\rangle\leq \langle a\rangle$.
\end{lem}
\begin{proof} Recall that the Frattini subgroup of 
$H$, denoted by $\Phi(H)$ is the set of all non-generators of $H$. Then, if 
$H=\langle a_1\rangle\times \dots \times \langle a_r\rangle$, it is easy to see that 
$\Phi(H)=\langle{a_1}^p\rangle\times \dots \times \langle{a_r}^p\rangle$. As $x$ is a non-generator, 
$x\in \Phi(H) $  and $x=({a_1}^p)^{i_1}\dots ({a_r}^p)^{i_r}$ where $0\leq i_j\leq |a_j|$, that is $x=({a_1}^{i_1}\dots {a_r}^{i_r} )^p$.
If at least one of $i_k$'s is not a multiple of $p$, set $a={a_1}^{i_1}\dots {a_r}^{i_r}$, then 
$a\notin \Phi(H) $ which means that $a$ is a generator of $H$ and $\langle a^p\rangle=\langle x\rangle\leq \langle a\rangle$.
If all $i_k'$s are multiple of $p$, then 
for $a={a_1}^{\frac{i_1}{p}}\dots{a_r}^{\frac{i_r}{p}} $ we have that 
$x=a^{p^2}$. So $\langle x\rangle \leq \langle a\rangle$. 
If one of $\frac{i_k}{p}$'s is not a multiple of $p$, then $a\notin \Phi(H) $ and $a$ is a generator. If all $\frac{i_k}{p}$'s are a multiple of $p$ then one can continue the process until getting an $\frac{i_k}{p}$ which is not a multiple of $p$
for some $1 \leq k \leq r$.
\end{proof}

\begin{lem}\label{generator} Let $G$ be an abelian $p$-group of exponent $p^n$. Assume that 
$H$ and $K$ are isomorphic subgroups of $G$ such that for some $x,y\in G$  of order $p^n$, we 
have that $$G/H=\langle xH\rangle\cong C_{p^m}\cong \langle yK\rangle=G/K$$ where $0 \leq m\leq n$. 
Then $x^{p^m}$ is a generator of $H$ if and only if $y^{p^m}$ is a generator of $K$.
 \end{lem}

\begin{proof} Suppose that $x^{p^m}$ is a generator of $H$, then $H=\langle x^{p^m}\rangle\times H_1$ 
for some $H_1 \leq H$. Suppose for a contradiction that $y^{p^m}$ is not a generator of $K$. Then  by Lemma 
\ref{nongenerator}, there exists a generator $a\in K$ such that $y^{p^m}=a^{p^t}$ for some $t$
with $t > 0$. Note that 
$K=\langle a\rangle\times K_1$ for some $K_1 \leq K$. On the other hand, 
since the orders of $x$ and $y$ are equal, we have that $\langle x^{p^m}\rangle \not\cong \langle a\rangle$. Hence, 
since $H\cong K$,
there exists $b\in H_1$ with $\langle b\rangle\cong \langle a\rangle$ and there exists $z\in K_1$ with 
$\langle z\rangle\cong \langle x^{p^m}\rangle$, such that $K=\langle a\rangle \times \langle z\rangle\times K_2$ 
and $H=\langle b\rangle\times \langle x^{p^m}\rangle\times H_2 $ where $K_2\cong H_2$. It follows that 
$$\langle a\rangle /\langle a^{p^t}\rangle \times  \langle z\rangle\times K_2 \cong K/\langle y^{p^m}\rangle\ncong H/\langle x^{p^m}\rangle\cong \langle b\rangle \times H_2 .$$ But this contradicts with Lemma \ref{isoquotient}. The converse implication is similar.
\end{proof}

\begin{pro} \label{Teta}Let $G$ be an abelian $p$-group of exponent $p^n$. Assume that $H$ and $K$ are 
isomorphic subgroups of $G$ such that for some  $x,y\in G$  of order $p^n$,  $G/H=\langle xH \rangle \cong C_{p^m}\cong \langle yK \rangle=G/K$
where $0 \leq m\leq n$. Then, there exists an isomorphism $\theta: H\rightarrow K$ such that $\theta(x^{p^m})=y^{p^m}$.
\end{pro}

\begin{proof} 

If $m=n$ we have that $x^{p^m}=y^{p^m}=1$, then for any isomorphism $\theta: H\rightarrow K$ we 
definitely have $\theta(x^{p^m})=y^{p^m}$.
So we can assume that $m<n$. 
Then we have that $x^{p^m}\in H$ and $y^{p^m}\in K$. There are two cases.\\
\textbf{Case 1:} $x^{p^m}$ is a generator of $H$. \\
By Lemma \ref{generator}, $y^{p^m}$ is also a generator  of $K$.   Then 
$H=\langle x^{p^m} \rangle \times H_1$ and $K=\langle y^{p^m} \rangle \times K_1$ where $ H_1\cong K_1$.  
So one can choose $\theta_1:H_1 \to K_1$ as an isomorphism, and
define $\theta: H\rightarrow K$ as 
$$\theta((x^{p^m})^l  h)=(y^{p^m})^l  \theta_1(h)$$ for $0\leq l \leq p^{n-m}-1$ and $h \in H_1$.
It is clear that $\theta$ is an isomorphism from $H$ onto $K$  which satisfies $\theta(x^{p^m})=y^{p^m}.$\\
\textbf{Case 2:} $x^{p^m}$ is a non-generator of $H$.\\
By Lemma \ref{generator},  $y^{p^m}$ is also a non-generator of $K$. In this case, by Lemma \ref{nongenerator}, 
there exists a generator $a\in H$ such that $\langle x^{p^m} \rangle \leq \langle a\rangle$. 
Similarly, there exists a generator $b\in K$ such that $\langle y^{p^m} \rangle \leq \langle b\rangle$. 
We claim that $\langle a\rangle \cong \langle b\rangle$.
Assume that $\langle a\rangle \ncong \langle b\rangle$. Then as $H\cong K$, we have that
$H=\langle a\rangle \times \langle b_1\rangle \times H_1$ and $K=\langle a_1\rangle \times \langle b \rangle \times K_1$ where $\langle a\rangle \cong \langle a_1\rangle$, $\langle b_1\rangle \cong \langle b\rangle $, $H_1\cong K_1$.  
Note that $\langle a_1\rangle \not \cong \langle b_1\rangle $ since $\langle a\rangle \not \cong \langle b\rangle $. 
We have that 
$$(\langle a\rangle  /\langle x^{p^m} \rangle) \times \langle b_1\rangle \times H_1 \cong H/\langle x^{p^m} \rangle\ncong K/\langle y^{p^m} \rangle \cong  \langle a_1\rangle \times  (\langle b\rangle / \langle y^{p^m} \rangle )\times K_1 .$$  
This contradicts with Lemma \ref{isoquotient}. So $\langle a\rangle \cong \langle b\rangle$. Hence $H=\langle a \rangle \times H_1$, $K=\langle b\rangle \times K_1$ where $H_1\cong K_1$. Choose $\theta_1: H_1 \to K_1$ 
as one of those isomorphisms. 
Then one can define $\theta:H\rightarrow K$ as 
$$\theta(a^i \cdot h)=b^i \cdot \theta_1(h)$$ for $0\leq i \leq |a|-1$ and $h \in H_1$. 
Since the orders of $a$ and $b$ are equal, now it is easy to see that 
$\theta$ is an isomorphism between $H$ and $K$ which takes $x^{p^m}$ to $y^{p^m}$.
\end{proof}

\begin{lem}\label{ccyclic} Assume that $G=H\times K$  where $(|H|,|K|)=1$. Then, we have that 
$G_1$ is a cocyclic subgroup of $G$ if and only if 
 $G_1=H_1\times K_1$ where $H_1$ is a cocyclic subgroup of $H$ and $K_1$ is a cocyclic subgroup of $K$.
\end{lem}

\begin{proof}
If $H_1$ is a cocyclic subgroup of $H$ and $K_1$ is a cocyclic subgroup of $K$, it is easy to see that 
$H_1\times K_1$ is a cocyclic subgroup of $H \times K$ since the orders of $H$ and $K$ are coprime. 
Conversely, if $G_1$ is a cocyclic subgroup of $H \times K$. Then, since $H$ and $K$ are groups 
of coprime order, we have that $G_1=H_1 \times K_1$ 
where $H_1 \leq H$ and $K_1 \leq K$. But since $G_1$ is a cocyclic subgroup of $H \times K$, we have that 
$H_1$ is a cocyclic subgroup of $H$ and $K_1$ is a cocyclic subgroup of $K$.
\end{proof}

Now we are ready to prove Proposition \ref{G-iso}.

\begin{proof}[Proof of Proposition \ref{G-iso}]  
Since any pair of $G$-isomorphic subgroups of $G$ are isomorphic by definition, it is enough to prove that 
isomorphic cocyclic subgroups of $G$ are $G$-isomorphic.
Let us prove this statement first for $p$-groups. Assume that 
$G$ is an abelian $p$-group of exponent $p^n$ and let $H$ and $K$ be two cocyclic isomorphic subgroups of $G$. 
If $H=K=G$ there is nothing to do. So assume that $H,K< G$. By Lemma \ref{forsome-x}, there exist $x,y\in G$ of order $p^n$ 
such that $G/H=\langle x H \rangle $ and $G/K=\langle y K \rangle $.
Since $H$ and $K$ are isomorphic, their cyclic quotient groups  $G/H$ and $G/K$ are isomorphic. Let
$G/H=\langle xH \rangle \cong C_{p^ m} \cong\langle yK \rangle =G/K $ for some $m$ with $0<m\leq n$. 
By using Proposition \ref{Teta}, 
 we can choose an isomorphism $\theta: H \to K$ so that $\theta (x^{p^m})=y^{p^m}$. 
For each $g \in G$ there exist a unique $i$ where $0 \leq i \leq p^m-1$ and
a unique $h \in H$ such that $g=x^i h$. Now, define $\varphi: G \to G$ as 
$\varphi(g)=y^i \theta(h)$. We claim that $\varphi$ 
is an automorphism of $G$. It is easy to see that $\varphi$ 
is a bijection. To show that it is a homomorphism let us choose $g_1=x^{i_1} h_1$ and 
$g_2=x^{i_2} h_2$ for $0 \leq i_1, i_2 < p^m$ and $h_1, h_2 \in H$. Note that 
$0 \leq i_1+i_2 < 2 p^m$, we will show that $\varphi$ is a homomorphism by 
considering two seperate cases depending on the value of this sum. \\
\textbf{Case 1:} $0 \leq i_1+i_2 <  p^m$\\
In this case $\varphi(g_1 g_2)=\varphi(x^{i_1+i_2} h_1 h_2)=y^{i_1+i_2} \theta(h_1 h_2)=(y^{i_1} \theta(h_1))(y^{i_2} \theta(h_2))=\varphi(g_1) \varphi(g_2).$\\
\textbf{Case 2:} $p^m \leq i_1+i_2 < 2 p^m$\\
In this case, we have that 
$$\varphi(g_1 g_2)=\varphi((x^{i_1+i_2-p^m}) (x^{p^m} h_1 h_2))=
y^{i_1+i_2-p^m} \theta(x^{p^m} h_1 h_2)=y^{i_1+i_2-p^m} \theta(x^{p^m}) \theta(h_1) \theta(h_2)$$ from the definition of $\varphi$
and from the fact that $x^{p^m} \in H$. But this last expression 
is equal to $y^{i_1+i_2} \theta(h_1 h_2)$ since $y^{p^m}=\theta(x^{p^m})$. Moreover 
$\varphi(g_1) \varphi(g_2)=(y^{i_1} \theta(h_1))(y^{i_2} \theta(h_2))=y^{i_1+i_2} \theta(h_1 h_2).$ 
Hence, $\varphi(g_1 g_2)=\varphi(g_1) \varphi(g_2)$ in this case, too.\\
In both of the cases, we have shown that $\varphi$ is an automorphism of $G$ and it is easy 
to observe that $\varphi$ takes $H$ onto $K$.
Hence, $H$ and $K$ are $G$-isomorphic. 

Now, let $G$ be a finite abelian group whose order is composite. For $1 \leq i \leq r$, 
let $G_{p_i}$ denote a Sylow $p_i$-subgroup of $G$. Then $G=G_{p_1}\times G_{p_2}\times ...\times G_{p_r}.$
If $H$ and $K$ are two cocyclic subgroups of $G$ which are isomorphic, 
$H=H_{p_1}\times H_{p_2}\times ...\times H_{p_r}$ and $K=K_{p_1}\times K_{p_2}\times ...\times K_{p_r}$ 
where for each $1\leq i \leq r$, one has $H_{p_i} \cong K_{p_i}$. Moreover, by Lemma \ref{ccyclic}, for each $i$, 
the groups $H_{p_i} $ and $K_{p_i}$ are cocyclic subgroups of $G_{p_i}$. Hence, by the first 
part of the proof, $H_{p_i} $ and $K_{p_i}$ are  $G_{p_i}$-isomorphic subgroups or equivalently 
there exists an automorphism $\varphi_i$ of $G_{p_i}$ which takes $H_{p_i}$ onto $K_{p_i}$. 
Now let us set $\varphi:=(\varphi_1, \varphi_2, \ldots, \varphi_r)$, then it is easy to see that $\varphi$ is an automorphism of $G$ 
which takes $H$ onto $K$.
\end{proof}

\begin{proof}[Proof of Theorem \ref{rewrite}]  
Follows from Proposition III.2, Proposition
III.7 and Proposition III.8 in \cite{FGM}, together with Proposition \ref{G-iso}.
\end{proof}

For the proofs of Theorem \ref{HxH}, Theorem \ref{HxK} and Theorem \ref{main} 
we need to consider direct products of groups whose orders are relatively prime.  

When $G$ is equal to the direct product of subgroups of coprime order, 
the number of isomorphism classes of cocyclic subgroups, hence the number of 
non-$G$-equivalent minimal abelian codes is calculated easily as follows.

\begin{thm}\label{numberorbit}Let $G=H\times K$  where $(|H|,|K|)=1$.
Then we have that $\eta(\FF G)=\eta(\FF K)\ \eta(\FF H). $
\end{thm}
\begin{proof} By Lemma \ref{ccyclic}, any cocylic subgroup of $G$ is of the form $H_1\times K_1$ where $H_1$ is a cocyclic of $H$ and $K_1$ is a cocyclic of $K$.
So the number of isomorphism classes of cocyclic subgroups of $G$ is the product of number of isomorphism classes of cocyclic subgroups  of $H$ and the number of isomorphism classes of cocyclic subgroups of $K$. 
Now the result follows from  Theorem \ref{rewrite}.
\end{proof}

Now, thanks to Theorem \ref{rewrite}, to count the number of non-$G$-equivalent minimal abelian codes over $\FF$, we 
just need to count the number of isomorphism classes of cocyclic subgroups of $G$.

\begin{proof}[Proof of Theorem \ref{HxH}]
By using classification of finitely generated abelian groups and Theorem \ref{numberorbit} it is enough to prove the result when $H$
is a finite $p$-group. Let $H = C_{p^{a_1}} \times \ldots \times C_{p^{a_n}}$ where $a_i \geq 1$ are integers and $G=H^k$ for $k\geq 1$,
then $G=G_1 \times \ldots \times G_n$ where
$G_i= (C_{p^{a_i}})^k $ for $i=1, \ldots, n$. Let $L$ be a cocyclic subgroup of $G$. For each $i$, we have that
$$G_i/G_i \cap L \cong G_i L / L \leq G/L$$
which implies that $G_i/G_i \cap L$ is cyclic. So $ G_i \cap L$ should contain a subgroup $L_i$
which is isomorphic to $(C_{p^{a_i}})^{k-1} $
(for example by \cite[Theorem V.2]{FGM} ). Moreover, it is easy to see that for each $i$,
  there exists an element
$x_i \in G_i $ of order $p^{a_i}$ such that $G_i= L_i \times \langle x_i \rangle$.
Hence
$$G=(\prod_{i=1}^n L_i) \times (\prod_{i=1}^n \langle x_i \rangle),$$ where the first
term of the product is isomorphic to $(H)^{k-1}$ and the second is isomorphic to $H$.
 Now, by the use of Correspondence Theorem, there is a bijection between the subgroups of $G$ containing
$ \prod_{i=1}^n L_i$ and the subgroups of $\prod_{i=1}^n \langle x_i \rangle$.
Under this bijection, $L$ corresponds to a cocyclic subgroup $C_L$ of $\prod_{i=1}^n \langle x_i \rangle$
where $L= (\prod_{i=1}^n L_i) \times C_L$. By Theorem \ref{rewrite}, the result follows since
$\prod_{i=1}^n \langle x_i \rangle$ is isomorphic to $H$.
\end{proof}

\begin{proof}[Proof of Theorem \ref{HxK}]
It is enough to prove the result when $H$ and $K$ are finite $p$-groups by the classification of finitely generated abelian
groups and Lemma \ref{numberorbit}. Let $p^n$ be the exponent of $H$ and $K$. Then there exists $x\in H$ 
of order $p^n$ such that $H=\langle x \rangle \times \hat{H}$ where $\hat{H}$ is a finite $p$-group
of exponent less or equal than $p^n$ and $K\cong(C_{p^n})^r$ for some positive integer $r$. 
Let $G_1\leq G=H\times K$ with $G_1\cong (C_{p^n})^{r+1}$ so that
$G=G_1 \times \hat{H}$ and let $L$ be a cocyclic subgroup of $G$. Then by a similar reasoning as 
in the proof of Theorem \ref{HxH} we deduce that $G_1 / G_1 \cap L$ is cyclic and so $G_1 \cap L$
contains a subgroup isomorphic to $(C_{p^n})^r$, call this subgroup as $K_L$.  Then there exists an 
element $x_1 \in G_1$ of order $p^n$ such that $G_1=K_L \times \langle x_1 \rangle $. So
$G=K_L \times \langle x_1 \rangle  \times \hat{H}$ and letting $H_L=\langle x_1 \rangle  \times \hat{H}, G$ is equal to $K_L \times H_L$
where $K_L$ and $H_L$ are isomorphic to $K$ and $H$, respectively. Since $L$ is a cocyclic subgroup of $K_L \times H_L$ containing
$K_L$, by the Correspondence Theorem $L$ corresponds to a cocyclic subgroup $C_L$ of $H_L$. Therefore $L=K_L \times C_L$ where
$K_L$ is isomorphic to $K$ and $C_L$ is cocyclic subgroup of $H_L$.  Now, the result follows from Theorem \ref{rewrite}.
\end{proof}

In Theorem \ref{HxK}, the assumption on the exponents of the groups is important. We end this section by 
presenting the significance of this assumption with the following examples.

\begin{example} For an odd prime $p$, if we take $H=C_p\times C_p$ and $K=C_{p^2}\times C_{p^2}$, then $\eta(\FF H)=2$, $\eta(\FF K)=3$ and $\eta(\FF(H\times K))=4$ where the characteristic of $\FF$ is coprime to $p$.

\end{example}

\begin{example}Let $K$ be a finite homocyclic group and $H$ be a finite abelian group which is not homocyclic and $exp(K)> exp(H)$. Let $\FF$ be a finite field of characteristic coprime to $3$. If $G=K \times H$,
$\eta(\FF G)=\eta(\FF H)$ is no longer true. Indeed, if $K=C_{27}\times C_{27}$ and $H=C_9\times C_3$, then $\eta(\FF H)=4$, $\eta(\FF K)=4$ and $\eta(\FF G)=8$.
However, for $G=C_{27}\times C_{27}\times C_9\times C_3$ if we write $G=H\times K$ where $K=C_{27}$ and $H=C_{27}\times C_9\times C_3 $ then $\eta(\FF G)=\eta(\FF H)$.
\end{example}
\begin{example}Consider  $G=C_{27}\times C_9\times C_3\times C_3$ and take $H=C_{27}\times C_9$
and $K=C_3\times C_3$. $\eta(\FF H)=6$, $\eta(\FF K)=2$ and $\eta(\FF G)=8$ where the characteristic of $\FF$ is coprime to $3$.
\end{example}

\section{Calculation for $C_{p^n} \times C_{p^m}$ }

For the proof of Theorem \ref{exactnumber}, we need the following lemma.

\begin{lem}\label{noncyclic}
Let $G=\langle a \rangle \times \langle b \rangle$ where $\langle a \rangle \cong C_{p^n}$ and 
$\langle b \rangle \cong C_{p^m}$ with $n > m\geq 1$.
Assume that $L$ is a cocyclic subgroup  of $G$ which is not cyclic. Then 
$L\cong C_{p^i} \times  C_{p^j}$
where  $m \leq i \leq n$
and  $1 \leq j \leq m$.
\end{lem}

\begin{proof} 
Let $L$ be such a cocyclic subgroup of $G$. If the exponent of $L$ is $p^n$, then $L\cong C_{p^n}\times C_{p^j}$ where $1\leq j\leq m$. 
If the exponent of $L$ is strictly less than $p^n$, 
by Lemma \ref{any-z}, $G/L=\langle a L \rangle$, so that $G=L \langle a \rangle$. Thus, we have that 
$$C_{p^m} \cong G/\langle a \rangle = L \langle a \rangle / \langle a \rangle \cong L/L\cap \langle a \rangle$$
that is $L$ has a quotient isomorphic to $C_{p^m}$, hence $L$ has a subgroup isomorphic to 
$C_{p^m}$. So the exponent of $L$ is at least $p^m$, in this case. Hence, $L \cong C_{p^i} \times C_{p^j}$ 
for $i\geq m$ and $j\geq 1$. As $|L|\geq p^{m+1}$, $|G/L|\leq p^{(m+n)-(m+1)}=p^{n-1}$.  Therefore, the index of $L$ in $G$ is at most $p^{n-1}$.

We prove the required result by induction on the index of the cocyclic subgroup $L$ in $G$. Clearly, the statement holds when $|G/L|=1$. If $|G/L|=p$, 
then either $L\cong C_{p^{n-1}}\times C_{p^m}$ or $L\cong C_{p^n}\times C_{p^{m-1}}$.
Assume the statement  holds for any non-cyclic cocyclic subgroup of $G$ with index strictly less than  $p^s$ where $1\leq s\leq n-1$. Now let $L$ be a cocyclic subgroup of $G$ such that $|G/L|=p^s$. Then there exists a cocyclic subgroup  $L_1$ of $G$ such that $L<L_1\leq G$ where $|G/L_1|=p^{s-1}$. 
By our induction hypothesis, $L_1\cong C_{p^i}\times C_{p^j}$ where $m\leq i \leq n$ and $1 \leq j \leq m$. Moreover, $i$ and 
$j$ satisfy $m+n-(i+j)=s-1 \leq n-2$.
We also have that $L_1/L\cong C_p$. Thus, if $i\neq m$ and $j \neq 1$, then we deduce that $L\cong C_{p^{i-1}}\times C_{p^j}$ 
or $L\cong C_{p^{i}}\times C_{p^{j-1}}$. 
If $i=m$, then $j \geq 2$ by the inequality $m+n-(i+j)=s-1 \leq n-2$.  Since the exponent of $L$ is at least $p^m$, 
we shall have that $L\cong C_{p^{m}}\times C_{p^{j-1}}$. 
If $j=1$, then by the same inequality we have that $i \geq m+1$.  As $L$ is not cyclic, we shall have that $L\cong C_{p^{i-1}}\times C_{p}$.
Therefore, we deduce that 
$L\cong C_{p^i}\times C_{p^j}$ where $m\leq i \leq n$ and $1 \leq j \leq m$.

\end{proof}

\begin{pro}\label{alllist}
Let $G=\langle a \rangle \times \langle b \rangle$ where $\langle a \rangle \cong C_{p^n}$ and 
$\langle b \rangle \cong C_{p^m}$ with $n > m\geq 1$.
Then any cocyclic subgroup is isomorphic to one of the 
following subgroups in the following set $$\{H_k\times K_j|~~ H_k=\langle a^{p^k}b \rangle, K_j=\langle b^{p^j}\rangle,\ 0\leq k \leq n-m, \ 0 \leq j \leq m\}.$$

\end{pro}

\begin{proof} There are two cases to consider.\\
\textbf{Case 1 (Cocyclic subgroups of $G$ which are cyclic):} For each $k\in \{0,\dots n-m\}$,  
$H_k$ is a cocyclic subgroup of $G$, because $G/H_k= \langle aH_k\rangle \cong C_{p^{m+k}}$.  
Notice that there are exactly $n-m+1$ such subgroups of $G$. Up to isomorphism, there is no other 
cyclic cocyclic subgroup of $G$. Indeed, if there is one such subgroup $H$ which 
is not isomorphic to any  $H_k$ for $0\leq k \leq n-m$, then $|H|=p^s$ where $s\in \{0\dots m-1\}$. In this 
case, $G/H\cong C_{{p}^{n+m-s}}$ where $n+m-s\geq n+1$, but this  is impossible since the exponent 
of $G$ is equal to $p^n$. 
Note that $H_k\cong H_k\times K_m$ because $K_m=1$.\\
\textbf{Case 2 (Cocyclic subgroups of $G$ which are not cyclic):} From Case 1, we know that $H_k$ is a cocyclic subgroup of $G$, 
for each $k$, so that 
$G/H_k$ is cyclic. Since for each $j$, the quotient $G/(H_k \times K_j)$ is isomorphic to a subgroup of $G/H_k$, 
we have that $H_k \times K_j$ is a cocyclic subgroup of $G$ 
for any $j\in \{0,\dots m\}$. Conversely, using Lemma 
\ref{noncyclic}, we deduce that any cocyclic subgroup is isomorphic to one of $H_k\times K_j$
since $H_k \cong C_{p^{n-k}}$ where $k\in \{0,\dots n-m,\}$ and  $K_j \cong C_{p^{m-j}}$ 
for $j\in \{0,\dots m-1\}$.
\end{proof}

\begin{proof}[Proof of Theorem \ref{exactnumber}] By Proposition \ref{alllist},
the number of  isomorphism classes of cocyclic subgroups of $G$ is $(n-m+1)(m+1)$.
By Theorem \ref{rewrite}, $\eta(\FF G)=(n-m+1)(m+1).$
\end{proof}

An immediate consequence of Theorem \ref{exactnumber} and Theorem \ref{HxK} is the following result. 

\begin{cor}Let $n,m $ and $s$ be positive integers such that $n>m$.
If $G=(C_{p^n}\times C_{p^m})^s$ for $s\in \mathbb{N}$, then $\eta(\FF G)=(n-m+1)(m+1)$. Moreover if $G=C_{p^n}\times C_{p^m}\times(C_{p^n})^s$,
then $\eta(\FF G)=(n-m+1)(m+1)$.
\end{cor}

\section{Proof of Theorem \ref{main}}

Let $\tau(G)$ denote the number of divisors of the exponent of $G$. It is not difficult to see that the 
number of non-$G$-equivalent minimal abelian codes is greater than or equal to $\tau(G)$ when $G$ is 
a finite abelian group. Therefore, if the exponent of $G$ is given, Theorem \ref{main} gives a complete characterization of the groups
having $\tau(G)$ non-$G$-equivalent minimal abelian codes, that is having the least possible number of non-$G$-equivalent 
abelian codes. For the proof of Theorem \ref{main},
first of all we find the number of non-$G$-equivalent minimal abelian group codes for homocyclic $p$-groups and prove
the following.

\begin{thm}\label{case:p}Let $G$ be a finite abelian $p$-group. The number 
of non-$G$-equivalent minimal abelian codes is equal to $\tau(G)$ if and 
only if $G$ is homocyclic.
\end{thm}

\begin{proof} Assume that $G$ is homocyclic, that is $G\cong (C_{p^n})^s$. Then by Theorem \ref{HxH}, $\eta(\FF G)=\eta(\FF C_{p^n})$. Now it is clear that the number of isomorphism classes of subgroups of $C_{p^n}$ is equal to the number of divisors of $p^n$.
For the converse implication, assume that $G$ is not homocyclic. If the exponent of $G$ is $p^r$ for some $r\geq 1$,
then $G\cong C_{p^r}\times H$ where $H\cong K\times C_{p^i}$ for some $1\leq i\leq r-1$ for some subgroup $K$ of $H$. 
Then $H, H\times C_p, H\times C_{p^2},...,H\times C_{p^{r-1}}$ is a family of non-isomorphic cocyclic subgroups of $G$.
Obviously $K\times C_{p^r}$ is another cocyclic subgroup which is not isomorphic to none of the elements of this family. 
So we have at least $r+2$ non-isomorphic cocyclic subgroups. Hence, $\eta(\FF G)$ is at least $r+2$.  This leads to a contradiction because $\tau(G)=r+1$.
\end{proof}

\begin{proof}[Proof of Theorem \ref{main}]
 Let $G=G_{p_1}\times G_{p_2}\times ...\times G_{p_k} $
 where each $G_{p_i}$ is a homocyclic Sylow $p_i$-subgroup of $G$. 
 If the exponent of each $G_{p_i}$ is $p^{e_i}$,  then by Theorem \ref{case:p},
 $\eta(\FF G_{p_i} )$ is equal to $\tau(G_{p_i})=e_i +1$.
 By Lemma \ref{numberorbit}, $\eta(\FF G)$ is equal to $\prod_{i=1}^{k}(e_i+1)$ which is equal to $\tau(G).$

For the converse implication, assume for some $i$, a Sylow $p_i$-subgroup $G_{p_i}$ is not homocyclic.
Then by Theorem \ref{case:p},  $\eta(\FF G_{p_i} )>\tau(G_{p_{i}})=e_i+ 1$ which gives a contradiction.
\end{proof}

\bigskip


{\bf{Acknowledgements.}}
{\small
The authors were partially supported by Mimar Sinan Fine Arts University 
Scientific Research Unit with project number 2019-27.
}

\bibliographystyle{amsplain}

\end{document}